\documentclass[12pt]{amsproc}
\usepackage{hyperref,xy,doi}
\xyoption{all}

\newtheorem{theorem}{Theorem}[section]
\newtheorem{corollary}[theorem]{Corollary}
\newtheorem{lemma}[theorem]{Lemma}
\theoremstyle{definition}
\newtheorem{definition}[theorem]{Definition}
\theoremstyle{remark}
\newtheorem{remark}[theorem]{Remark}

\DeclareMathOperator{\Res}{Res}
\DeclareMathOperator{\Sym}{Sym}
\DeclareMathOperator{\Hom}{Hom}
\DeclareMathOperator{\Rep}{Rep}

\DeclareMathOperator{\wt}{wt}

\newcommand{\sgn}{\mathrm{sgn}}
\newcommand{\tr}{\mathrm{trace}}
\newcommand{\mom}[1]{\left\langle #1\right\rangle}

\newcommand{\fimod}{\FI\mathrm{-Mod}}
\newcommand{\FI}{\mathrm{FI}}
\newcommand{\ev}{\mathrm{ev}}
\renewcommand{\Vec}{\mathrm{Vec}}
\newcommand{\F}{\mathcal F}
\newcommand{\mch}[2]{\ensuremath{\left(\kern-.3em\left(\genfrac{}{}{0pt}{}{#1}{#2}\right)\kern-.3em\right)}}

\newcommand{\Par}{\mathrm{Par}}
\newcommand{\vv}{\mathbf v}
\newcommand{\ZZ}{\mathbf Z}

\title{Character Polynomials and the Restriction Problem}
\date{14 April 2021}

\author[Narayanan]{Sridhar P. Narayanan}
\address{The Institute of Mathematical Sciences (HBNI), Chennai}
\email{sridharn@imsc.res.in}

\author[Paul]{Digjoy Paul}
\address{Tata Institute of Fundamental Research, Mumbai}
\email{digjoypaul@gmail.com}

\author[Prasad]{Amritanshu Prasad}
\address{The Institute of Mathematical Sciences (HBNI), Chennai}
\email{amri@imsc.res.in}

\author[Srivastava]{Shraddha Srivastava}
\address{Uppsala University, Uppsala}
\email{maths.shraddha@gmail.com}

\begin{document}
\begin{abstract}
  Character polynomials are used to study the restriction of a polynomial representation of a general linear group to its subgroup of permutation matrices.
  A simple formula is obtained for computing inner products of class functions given by character polynomials.
  Character polynomials for symmetric and alternating tensors are computed using generating functions with Eulerian factorizations.
  These are used to compute character polynomials for Weyl modules, which exhibit a duality.
  By taking inner products of character polynomials for Weyl modules and character polynomials for Specht modules, stable restriction coefficients are easily computed.
  Generating functions of dimensions of symmetric group invariants in Weyl modules are obtained.
  Partitions with two rows, two columns, and hook partitions whose Weyl modules have non-zero vectors invariant under the symmetric group are characterized.
  A reformulation of the restriction problem in terms of a restriction functor from the category of strict polynomial functors to the category of finitely generated FI-modules is obtained.
\end{abstract}
\maketitle
\section{Introduction}
\label{sec:introduction}
Let $K$ be a field of characteristic $0$.
Let $P=K[X_1,X_2,\dotsc]$, a ring of polynomials in infinitely many variables.
Regard $P$ as a graded algebra where the variable $X_i$ has degree $i$.
In this grading, the monomial $X_1^{a_1}X_2^{a_2}\dotsb $ has degree $\sum_i i a_i$.
\begin{definition}
For each $n\geq 1$, let $V_n$ be a representation of the symmetric group $S_n$.
The collection $\{V_n\}_{n=1}^\infty$ is said to have \emph{eventually polynomial character}, if there exists $q\in P$ and a positive integer $N$ such that, for each $n\geq N$ and each $w\in S_n$,
\begin{displaymath}
  \tr(w;V_n) = q(X_1(w),X_2(w),\dotsc),
\end{displaymath}
where $X_i(w)$ is the number of $i$-cycles in $w$.
The collection $\{V_n\}$ is said to have \emph{polynomial character} if $N=1$.
The polynomial $q$ is called the \emph{character polynomial} of $\{V_n\}$.
\end{definition}
Character polynomials have been used to study characters of families of representations of symmetric groups that occur naturally in combinatorics, topology and other areas.
A survey of their history can be found in the article of Garsia and Goupil \cite{MR2576382}.
More recently, Church, Ellenberg and Farb \cite{MR3357185} developed the theory of FI-modules.
They showed that each finitely generated FI-module gives rise to a family of representations with eventually polynomial character.

Any polynomial $q\in P$ gives rise to a class function on $S_n$ for every positive integer $n$.
The value of this function at $w\in S_n$ is obtained by substituting for $X_i$ the number of $i$-cycles in $w$.
For each $n$, we define the moment of $q$ as the average value of the associated class function on $S_n$.
The ring $P$ has a basis indexed by integer partitions, which we call the binomial basis (Definition~\ref{definition:binomial-basis}).
We give an explicit formula for the moment of a binomial basis element (Theorem~\ref{theorem:moment}).
This formula can be used to compute inner products of class functions coming from character polynomials.
It implies that such an inner product achieves a constant value for large $n$ (Corollary \ref{corollary:stability}).
This is a character-theoretic analogue of \cite[Theorem~1.13]{MR3357185}, which establishes representation stability for finitely generated FI-modules.

The homogeneous polynomial representations of degree $d$ general linear groups $GL_n(K)$ and the representations of the symmetric group $S_d$ are linked via Schur-Weyl duality (see \cite{Green}).
The restriction problem explores a different relationship between polynomial representation of $GL_n(K)$ and $S_n$.
It asks how an irreducible polynomial representation of $GL_n(K)$ decomposes when it is restricted to $S_n$ sitting inside $GL_n(K)$ as the subgroup of permutation matrices.

For each partition $\lambda$, let $W_\lambda$ denote the Weyl functor (see~\cite[Definition II.1.3]{ABW}) associated to $\lambda$.
Let $P(n,d)$ denote the set of all partitions of $d$ with at most $n$ parts.
Then $W_\lambda(K^n)$, as $\lambda$ runs over $P(n,d)$, are the irreducible polynomial representations of the general linear group $GL_n(K)$ of degree $d$.

Following the notation and terminology of \cite[Definition~2.2.5]{MR3357185}, for a partition $\mu=(\mu_1,\dotsc,\mu_m)$ of size $|\mu|$ and an integer $n\geq \mu_1+|\mu|$, let $\mu[n]$ denote the padded partition $(n-|\mu|, \mu_1,\mu_2,\dotsc,\mu_m)$.
Let $V_{\mu[n]}$ denote the Specht module of $S_n$ corresponding to $\mu[n]$.

Consider the decomposition of the restriction of $W_\lambda(K^n)$ to $S_n$ into Specht modules:
\begin{displaymath}
  \Res^{GL_n(K)}_{S_n} W_\lambda(K^n) = \bigoplus_\mu V_{\mu[n]}^{\oplus r_{\lambda\mu}(n)},
\end{displaymath}
where the sum is over partitions $\mu$ such that $n-|\mu|\geq \mu_1$.
It is well-known that the coefficients $r_{\lambda\mu}(n)$ are eventually constant for large $n$ (this result is attributed to D.~E.~Littlewood by Assaf and Speyer \cite{assaf-speyer}).
Let $r_{\lambda\mu}$ be their eventually constant value, which is called the \emph{stable restriction coefficient}.
Finding a combinatorial interpretation of $r_{\lambda\mu}$ is known as the \emph{restriction problem}.

In this article we show that the family $\{\Res^{GL_n(K)}_{S_n} W_\lambda(K^n)\}$ has polynomial character.
We determine its character polynomial $S_\lambda$ (Theorem~\ref{theorem:char-poly-weyl}) by applying the Jacobi--Trudi identities to the character polynomials of symmetric and exterior powers of $K^n$ (Theorem~\ref{theorem:he}).
The character polynomials of symmetric and exterior tensor powers of $K^n$ have generating functions with Eulerian factorization (Theorem~\ref{theorem:he}).
Multiplying $S_\lambda$ by the character polynomial $q_\mu$ of Specht modules $\{V_{\mu[n]}\}$ (which was computed by Macdonald \cite[Example~I.7.14(b)]{MR3443860} and Garsia--Goupil \cite{MR2576382}), and then taking moments (Theorem~\ref{theorem:moment}) gives an algorithm to compute stable restriction coefficients (Theorem~\ref{theorem:stable-res}).
Assaf and Speyer \cite{assaf-speyer} and independently, Orellana and Zabrocki \cite{2016arXiv160506672O} introduced \emph{Specht symmetric functions} to study the restriction problem.
In Section~\ref{sec:oz} we explain the relationship between these two approaches.

Notwithstanding several interesting recent developments \cite{assaf-speyer,2018arXiv180404702H,2016arXiv160506672O,2019arXiv190100378O}, a solution to the restriction problem remains elusive.
Even $r_{\lambda\emptyset}$ (here $\emptyset$ denotes the empty partition of $0$, so $r_{\lambda\emptyset}$ is the dimension of the space of $S_n$-invariant vectors in $W_\lambda(K^n)$ for large $n$) appears to be a non-trivial and interesting problem.
We provide generating functions in $\lambda$ for the dimension of the space of $S_n$-invariant vectors in $W_\lambda(K^n)$ (Corollary~\ref{corollary:weyl-mom}).
Using our main generating function (Theorem~\ref{theorem:hlaemu}) for the dimension of $S_n$-invariants in mixed tensors, we are able to characterize partitions with two rows, two columns and hook partitions which have non-zero $S_n$-invariant vectors.
To the best of our knowledge, the specific problem of determining the $S_n$-invariant vectors of a Weyl-module has not been studied before.

We conclude this paper by placing the restriction problem in the context of strict polynomial functors and FI-modules.
Friedlander and Suslin \cite{MR1427618} introduced strict polynomial functors of degree $d$.
The polynomial representations of degree $d$ of $GL_n(K)$ are obtained by evaluating strict polynomial functors of degree $d$ at $K^n$.
Similarly, families of representations of $S_n$ with stability properties can be obtained by evaluating finitely generated FI-modules at $\{1,\dotsc,n\}$ (see \cite{MR3357185}).
We define a functor from the category of strict polynomial functors of degree $d$ to the category of finitely generated FI-modules for every $d$ (Section~\ref{sec:restriction-functor}).
This functor corresponds to restriction of representations from $GL_n(K)$ to $S_n$ under evaluation functors (Theorem~\ref{theorem:Res}).
\section{Character Polynomials and their Moments}
\label{sec:char-polyn}
\subsection{Moments and Stability}
\label{sec:moments}
\begin{definition}[Moment]
  The moment of $q\in P$ at $n$ is defined as:
  \begin{displaymath}
    \mom q_n=\frac 1{n!} \sum_{w\in S_n} q(X_1(w),X_2(w),\dotsc).
  \end{displaymath}
\end{definition}

We shall express integer partitions in \emph{exponential notation}: given a partition $\alpha$ with largest part $r$, we write:
\begin{displaymath}
  \alpha = 1^{a_1}2^{a_2}\dotsb r^{a_r},
\end{displaymath}
where $a_i$ is the number of parts of $\alpha$ of size $i$ for each $1\leq i\leq r$.
Thus $\alpha$ is a partition of the integer $|\alpha|:=a_1+2a_2+\dotsb+ra_r$.
For every integer partition $\alpha=1^{a_1}\dotsb r^{a_r}$ define $\binom X\alpha\in P$ by:
\begin{displaymath}
  \binom X\alpha = \binom{X_1}{a_1}\binom{X_2}{a_2}\dotsb \binom{X_r}{a_r}.
\end{displaymath}
\begin{definition}
  [Binomial basis]
  \label{definition:binomial-basis}
  The basis of $P$ consisting of elements
  \begin{displaymath}
    \left\{\left.\binom X\alpha \;\right|\; \alpha \text{ is an integer partition}\right\}
  \end{displaymath}
  is called the \emph{binomial basis} of $P$.
\end{definition}
For an integer partition $\alpha=1^{a_1}2^{a_2}\dotsb r^{a_r}$, define $z_\alpha = \prod_{i=1}^r i^{a_i}a_i!$.
This is the order of the centralizer in $S_n$ of a permutation with cycle-type $\alpha$.
\begin{theorem}
  \label{theorem:moment}
  For every integer partition $\alpha = 1^{a_1}2^{a_2}\dotsb$, we have:
  \begin{displaymath}
    \mom{\binom X\alpha}_n =
    \begin{cases}
      0 & \text{ if } n < |\alpha|,\\
      1/z_\alpha & \text{ otherwise.}
    \end{cases}
  \end{displaymath}
\end{theorem}
\begin{proof}
  We have:
  \begin{displaymath}
    \sum_{n\geq 0}\mom{\binom X\alpha}_nv^n = \sum_{n\geq 0}\frac 1{n!}\sum_{w\in S_n}\prod_{i\geq 1} \binom{X_i(w)}{a_i}v^{iX_i(w)}.
  \end{displaymath}
  Replace the sum $w\in S_n$ by a sum over conjugacy classes in $S_n$.
  If $\beta=1^{b_1}2^{b_2}\dotsb$ is a partition of $n$, then the number of elements in $S_n$ with cycle type $\beta$ is $\frac{n!}{\prod_i i^{b_i}b_i!}$.
  We get:
  \begin{align*}
    \sum_{n\geq 0}\mom{\binom X\alpha}_nv^n & = \sum_{n\geq 0}\sum_{\beta\vdash n}\prod_{i\geq 1}\frac{v^{ib_i}}{i^{b_i}b_i!}\binom{b_i}{a_i}\\
                                            & = \sum_{b_i\geq a_i} \prod_{i\geq 1}\frac{v^{ib_i}}{i^{b_i}b_i!}\binom{b_i}{a_i}\\
                                            & = \sum_{b_i\geq a_i} \prod_{i\geq 1}\frac{v^{ib_i}}{i^{b_i}a_i!(b_i-a_i)!}\\
    & = \sum_{b_i\geq a_i} \prod_{i\geq 1}\frac{v^{ia_i}}{i^{a_i}a_i!} \frac{v^{i(b_i-a_i)}}{i^{b_i-a_i}(b_i-a_i)!}.
  \end{align*}
  Setting $c_i=b_i-a_i$ gives:
  \begin{align*}
    \sum_{n\geq 0}\mom{\binom X\alpha}_nv^n & = \frac{v^{|\alpha|}}{z_\alpha}\prod_{i\geq 1}\sum_{c_i\geq 0}\frac{v^{ic_i}}{i^{c_i}c_i!}\\
                                            & = \frac{v^{|\alpha|}}{z_\alpha}\sum_{n\geq 0} v^n\sum_{\gamma\vdash n} \frac 1{z_\gamma}.\\
  \end{align*}
  Since $\sum_{\gamma\vdash n} 1/z_\gamma = 1$ for every $n$, we get:
  \begin{equation}
    \label{eq:binomial-moments}
    \sum_{n\geq 0}\mom{\binom X\alpha}_nv^n  = \frac{v^{|\alpha|}}{z_\alpha}\frac 1{1-v},
  \end{equation}
  from which Theorem~\ref{theorem:moment} follows.
\end{proof}
For two representations $V$ and $W$ of $S_n$, let:
\begin{displaymath}
  \langle V, W\rangle_n = \dim \Hom_{S_n}(V,W),
\end{displaymath}
which is the same as the Schur inner product of their characters:
\begin{displaymath}
  \langle V, W\rangle_n = \frac 1{n!}\sum_{w\in S_n} \tr(w;V)\tr(w,W).
\end{displaymath}
\begin{corollary}
  \label{corollary:stability}
  For any $q\in P$ of degree $d$, $\mom{q}_n=\mom{q}_d$ for all $n\geq d$.
  In particular, if $\{V_n\}$ and $\{W_n\}$ are families of representations with polynomial characters of degree $d_1$ and $d_2$, then $\langle V_n,W_n\rangle_n$ stabilizes for $n\geq d_1+d_2$.
\end{corollary}
\begin{proof}
  This follows from the fact that the polynomials $\binom X\alpha$, as $\alpha$ runs over the set of integer partitions, form a basis of $P$.
\end{proof}
\begin{definition}
  [Stable moment]
  For a polynomial $q\in P$ we define the stable moment $\mom q$ of $q$ to be the eventually constant value of $\mom q_n$:
  \begin{displaymath}
    \mom q = \lim_{n\to \infty} \mom q_n.
  \end{displaymath}
\end{definition}
Let $V_n=V_{\lambda[n]}$, the Specht module of $S_n$ corresponding to the padded partition $\lambda[n]$.
It is well-known that $\{V_n\}$ is a family of representations with eventually polynomial character \cite[Proposition~I.1]{MR2576382}.
In other words, for every partition $\lambda$, there exists a polynomial $q_\lambda\in P$ such that
\begin{equation}
  \label{eq:qla}
  \chi_{\lambda[n]}(w) = q_\lambda(X_1(w),X_2(w),\dotsc) \text{ for } n\geq |\lambda|+\lambda_1,
\end{equation}
where $\chi_{\lambda[n]}$ denotes the character of the Specht module $V_{\lambda[n]}$.
Given three partitions $\lambda$, $\mu$, and $\nu$ of the same integer $k$, let $g_{\lambda\mu\nu}(n)$ denote the multiplicity of $V_{\lambda[n]}$ in $V_{\mu[n]}\otimes V_{\nu(n)}$.
Then
\begin{displaymath}
  g_{\lambda\mu\nu}(n) = \mom{q_\lambda q_\mu q_\nu}_n.
\end{displaymath}
By Corollary~\ref{corollary:stability}, $g_{\lambda\mu\nu}(n)$ is eventually constant, recovering a well-known theorem of Murnaghan (see \cite{MR95209}).
Church, Ellenberg, and Farb \cite[Section~3.4]{MR3357185} point out that this result can also be obtained by showing that the families $V_{\mu[n]}\otimes V_{\nu[n]}$ and $V_{\lambda[n]}$ come from finitely generated $\FI$-modules.

\subsection{Symmetric and Alternating Tensors}
\label{sec:symk-altk}
Let $\Sym^d$ and $\wedge^d$ denote the symmetric and alternating tensor functors respectively.
For every $n\geq 0$, $\Sym^d(K^n)$ and $\wedge^d(K^n)$ can be regarded as representations of $S_n$.
In this section, we will prove that they have polynomial character by direct computation.

We shall work with generating functions that live in the ring $P[[t]]$ of Laurent series in the variable $t$ and coefficients in $P$.
Let $\mch{X_i}j = \binom{X_i+j-1}j$.
The ring $P[[t]]$ admits expressions of the kind
\begin{align}
  \label{eq:pow1}
  (1-t^i)^{-X_i} & = \sum_{j\geq 0} \mch{X_i}j t^{ij},\\
  \label{eq:pow2}
  (1+(-t)^i)^{X_i} & = \sum_{j\geq 0} (-1)^{ij}\binom{X_i}j t^{ij},
\end{align}
which will be needed later.
Also, given $w\in S_n$ for any $n$ and $q\in P$, we write $q(w)=q(X_1(w),X_2(w),\dotsc)$.
\begin{theorem}
  \label{theorem:he}
  Let $\{H_d\}_{d=0}^\infty$ be the sequence of polynomials in $P$ defined by:
  \begin{equation}
    \label{eq:H}
    \sum_{d=0}^\infty H_dt^d = \prod_{i=1}^\infty(1-t^i)^{-X_i},
  \end{equation}
  an identity in the formal power series ring $P[[t]]$.
  
  As a consequence, for every $n\geq 1$ and every $w\in S_n$,
  \begin{displaymath}
    H_d(w) = \tr(w;\Sym^d(K^n)).
  \end{displaymath}
  Let $\{E_d\}_{d=0}^\infty$ be the sequence of polynomials in $P$ defined by:
  \begin{equation}
    \label{eq:E}
    \sum_{d=0}^\infty E_dt^d = \prod_{i=1}^\infty(1-(-t)^i)^{X_i}.
  \end{equation}
  Then for every $n\geq 1$ and every $w\in S_n$,
  \begin{displaymath}
    E_d(w) = \tr(w;\wedge^d(K^n)).
  \end{displaymath}
  For every positive integer $d$, we have:
  \begin{align}
    \label{eq:sym-char}
    H_d&=\sum_{\alpha\vdash d}\quad\prod_{i=1}^d\mch{X_i}{a_i},\\
    \label{eq:alt-char}
    E_d&=\sum_{\alpha\vdash d}\quad(-1)^{a_2+a_4+\dotsb}\prod_{i=1}^d\binom{X_i}{a_i}.
  \end{align}
\end{theorem}
\begin{proof}
  The standard basis of $K^n$ is indexed by the set $[n]=\{1,\dotsc,n\}$.
  The space $\Sym^d K^n$ has an induced basis indexed by multisets of size $d$ with elements drawn from $[n]$.
  The trace of $w\in S_n$ on $\Sym^d K^n$ is the number of such multisets that are fixed by $w$.
  In a multiset that is fixed by $w$, the elements in each cycle of $w$ appear with the same multiplicity.
  Hence a multiset fixed by $w$ may be regarded as a multiset of cycles of $w$.
  Assign weight $\wt(C)=t^i$ to each $i$-cycle $C$ of $w$.
  To a multiset $M = \{C_1^{m_1}\dotsb C_r^{m_r}\}$ of cycles of $w$, assign weight $\wt(M)=\prod_{1\leq j\leq r} \wt(C_j)^{m_j}$.
  Then
  \begin{displaymath}
    \sum_{d\geq 0} \tr(w; \Sym^d K^n)t^d  = \sum_M \wt(M),
  \end{displaymath}
  where the sum runs over all multisets $M$ of cycles of $w$.
  Each of the $X_i(w)$ $i$-cycles of $w$ contributes a factor $(1-t^i)^{-1}$ to this generating function, so
  \begin{displaymath}
    \sum_{d\geq 0} \tr(w; \Sym^d K^n)t^d  = \prod_{i\geq 1} (1-t^i)^{-X_i(w)}.
  \end{displaymath}

  Similarly, $\wedge^d K^n$ has a basis indexed by subsets of $[n]$ with $d$ elements.
  Only subsets fixed by $w$ contribute to the trace, and these are unions of the cycles of $w$.
  A cycle of length $i$ changes the sign of the corresponding basis vector by a factor of $-(-1)^i$.
  Assign weight $\wt(C)=-(-t)^i$ to each $i$-cycle $C$ of $w$.
  To a subset $N = \{C_1,\dotsc,C_r\}$ of cycles of $w$, assign weight $\wt(N) = \prod_{1\leq i\leq r} \wt(C_i)$.
  Then
  \begin{displaymath}
    \sum_{d\geq 0} \tr(w; \wedge^d K^n)t^d  = \sum_N \wt(N),
  \end{displaymath}
  where the sum runs over all subsets $N$ of the set of cycles of $w$.
  Each of the $X_i(w)$ $i$-cycles of $w$ contributes a factor $(1-(-t)^i)$ to this generating function, so
  \begin{displaymath}
    \sum_{d\geq 0} \tr(w; \wedge^d K^n)t^d  = \prod_{i\geq 1} (1-(-t)^i)^{X_i(w)}.
  \end{displaymath}
  Expansion of the products in \eqref{eq:H} and \eqref{eq:E} using (\ref{eq:pow1}) and (\ref{eq:pow2}) gives (\ref{eq:sym-char}) and (\ref{eq:alt-char}) respectively.
\end{proof}

\subsection{Character Polynomials of Weyl Modules}
\label{sec:char-polyn-weyl}
Applying the Jacobi--Trudi identities \cite[Section~I.3]{MR3443860} to the character polynomials of $\Sym^d$ and $\wedge^d$ gives character polynomials for Weyl functors.
For a partition $\lambda$, let $\lambda'$ denote its conjugate partition.
\begin{theorem}
  \label{theorem:char-poly-weyl}
  For every partition $\lambda$, the element of $P$ defined by
  \begin{equation}
    \label{eq:jt}
    S_\lambda = \det(H_{\lambda_i+j-i}) = \det(E_{\lambda'_i+j-i})
  \end{equation}
  is such that for every positive integer $n$ and every $w\in S_n$,
  \begin{displaymath}
    S_\lambda(w) = \tr(w;W_\lambda(K^n)).
  \end{displaymath}
\end{theorem}
\small{
\begin{table}[htbp]
  \def\arraystretch{1.3}
  \begin{center}
    \begin{tabular}{|p{.1\textwidth}|p{.8\textwidth}|}
      \hline
      $\lambda$ & $S_\lambda$\\
      \hline
      \multicolumn{2}{|c|}{$n= 1$}\\
      \hline
      $(1)$&$ X_{1} $\\
      \hline
      \multicolumn{2}{|c|}{$n=2$}\\
      \hline
      $(2)$&$ \frac{1}{2} X_{1}^{2} + \frac{1}{2} X_{1} + X_{2} $\\
      $(1^2)$&$ \frac{1}{2} X_{1}^{2} - \frac{1}{2} X_{1} -  X_{2} $\\
      \hline
      \multicolumn{2}{|c|}{$n=3$}\\
      \hline
      $(3)$&$ \frac{1}{6} X_{1}^{3} + \frac{1}{2} X_{1}^{2} + X_{1} X_{2} + \frac{1}{3} X_{1} + X_{3} $\\
      $(2, 1)$&$ \frac{1}{3} X_{1}^{3} - \frac{1}{3} X_{1} -  X_{3} $\\
      $(1^3)$&$ \frac{1}{6} X_{1}^{3} - \frac{1}{2} X_{1}^{2} -  X_{1} X_{2} + \frac{1}{3} X_{1} + X_{3} $\\
      \hline
      \multicolumn{2}{|c|}{$n=4$}\\
      \hline
      $(4)$&$ \frac{1}{24} X_{1}^{4} + \frac{1}{4} X_{1}^{3} + \frac{1}{2} X_{1}^{2} X_{2} + \frac{11}{24} X_{1}^{2} + \frac{1}{2} X_{1} X_{2} + \frac{1}{2} X_{2}^{2} + X_{1} X_{3} + \frac{1}{4} X_{1} + \frac{1}{2} X_{2} + X_{4} $\\
      $(3,1)$&$ \frac{1}{8} X_{1}^{4} + \frac{1}{4} X_{1}^{3} + \frac{1}{2} X_{1}^{2} X_{2} - \frac{1}{8} X_{1}^{2} - \frac{1}{2} X_{1} X_{2} - \frac{1}{2} X_{2}^{2} - \frac{1}{4} X_{1} - \frac{1}{2} X_{2} -  X_{4} $\\
      $(2,2)$&$ \frac{1}{12} X_{1}^{4} - \frac{1}{12} X_{1}^{2} + X_{1} X_{2} + X_{2}^{2} -  X_{1} X_{3} $\\
      $(2,1^2)$&$ \frac{1}{8} X_{1}^{4} - \frac{1}{4} X_{1}^{3} - \frac{1}{2} X_{1}^{2} X_{2} - \frac{1}{8} X_{1}^{2} - \frac{1}{2} X_{1} X_{2} - \frac{1}{2} X_{2}^{2} + \frac{1}{4} X_{1} + \frac{1}{2} X_{2} + X_{4} $\\
      $(1^4)$&$ \frac{1}{24} X_{1}^{4} - \frac{1}{4} X_{1}^{3} - \frac{1}{2} X_{1}^{2} X_{2} + \frac{11}{24} X_{1}^{2} + \frac{1}{2} X_{1} X_{2} + \frac{1}{2} X_{2}^{2} + X_{1} X_{3} - \frac{1}{4} X_{1} - \frac{1}{2} X_{2} -  X_{4} $\\
      \hline
      \multicolumn{2}{|c|}{$n=5$}\\
      \hline
      $(5)$&$ \frac{1}{120} X_{1}^{5} + \frac{1}{12} X_{1}^{4} + \frac{1}{6} X_{1}^{3} X_{2} + \frac{7}{24} X_{1}^{3} + \frac{1}{2} X_{1}^{2} X_{2} + \frac{1}{2} X_{1} X_{2}^{2} + \frac{1}{2} X_{1}^{2} X_{3} + \frac{5}{12} X_{1}^{2} + \frac{5}{6} X_{1} X_{2} + \frac{1}{2} X_{1} X_{3} + X_{2} X_{3} + X_{1} X_{4} + \frac{1}{5} X_{1} + X_{5} $\\
      $(4, 1)$&$ \frac{1}{30} X_{1}^{5} + \frac{1}{6} X_{1}^{4} + \frac{1}{3} X_{1}^{3} X_{2} + \frac{1}{6} X_{1}^{3} + \frac{1}{2} X_{1}^{2} X_{3} - \frac{1}{6} X_{1}^{2} - \frac{1}{3} X_{1} X_{2} - \frac{1}{2} X_{1} X_{3} -  X_{2} X_{3} - \frac{1}{5} X_{1} -  X_{5} $\\
      $(3, 2)$&$ \frac{1}{24} X_{1}^{5} + \frac{1}{12} X_{1}^{4} + \frac{1}{6} X_{1}^{3} X_{2} - \frac{1}{24} X_{1}^{3} + \frac{1}{2} X_{1}^{2} X_{2} + \frac{1}{2} X_{1} X_{2}^{2} - \frac{1}{2} X_{1}^{2} X_{3} - \frac{1}{12} X_{1}^{2} - \frac{1}{6} X_{1} X_{2} + \frac{1}{2} X_{1} X_{3} + X_{2} X_{3} -  X_{1} X_{4} $\\
      $(3, 1^2)$&$ \frac{1}{20} X_{1}^{5} - \frac{1}{4} X_{1}^{3} -  X_{1}^{2} X_{2} -  X_{1} X_{2}^{2} + \frac{1}{5} X_{1} + X_{5} $\\
      $(2^2, 1)$&$ \frac{1}{24} X_{1}^{5} - \frac{1}{12} X_{1}^{4} - \frac{1}{6} X_{1}^{3} X_{2} - \frac{1}{24} X_{1}^{3} + \frac{1}{2} X_{1}^{2} X_{2} + \frac{1}{2} X_{1} X_{2}^{2} - \frac{1}{2} X_{1}^{2} X_{3} + \frac{1}{12} X_{1}^{2} + \frac{1}{6} X_{1} X_{2} - \frac{1}{2} X_{1} X_{3} -  X_{2} X_{3} + X_{1} X_{4} $\\
      $(2, 1^3)$&$ \frac{1}{30} X_{1}^{5} - \frac{1}{6} X_{1}^{4} - \frac{1}{3} X_{1}^{3} X_{2} + \frac{1}{6} X_{1}^{3} + \frac{1}{2} X_{1}^{2} X_{3} + \frac{1}{6} X_{1}^{2} + \frac{1}{3} X_{1} X_{2} + \frac{1}{2} X_{1} X_{3} + X_{2} X_{3} - \frac{1}{5} X_{1} -  X_{5} $\\
      $(1^5)$&$ \frac{1}{120} X_{1}^{5} - \frac{1}{12} X_{1}^{4} - \frac{1}{6} X_{1}^{3} X_{2} + \frac{7}{24} X_{1}^{3} + \frac{1}{2} X_{1}^{2} X_{2} + \frac{1}{2} X_{1} X_{2}^{2} + \frac{1}{2} X_{1}^{2} X_{3} - \frac{5}{12} X_{1}^{2} - \frac{5}{6} X_{1} X_{2} - \frac{1}{2} X_{1} X_{3} -  X_{2} X_{3} -  X_{1} X_{4} + \frac{1}{5} X_{1} + X_{5} $\\
      \hline
    \end{tabular}
  \end{center}
  \caption{Character polynomials of Weyl modules}
  \label{table-char_poly}
\end{table}
}
The polynomials $S_\lambda$ for partitions $\lambda$ of integers at most $5$ are given in Table~\ref{table-char_poly}.
The highest degree coefficients in these expansions are character values of symmetric groups. More precisely, we have:
\begin{theorem}
  \label{theorem:binomial-coeff}
  Let $\lambda$ be a partition of a positive integer $d$.
  For every partition $\alpha=1^{a_1}2^{a_2}\dotsb$ of $d$, the coefficient of $\binom X\alpha$ in the expansion of $S_\lambda$ in the binomial basis (Definition~\ref{definition:binomial-basis}) is $\chi_\lambda(w_\alpha)$, where $w_\alpha$ is a permutation with cycle type $\alpha$.
\end{theorem}
The theorem will be a consequence of the following lemma:
\begin{lemma}
  \label{lemma:highest-coeffs}
  Let $\lambda$ be a partition of a positive integer $d$.
  For every partition $\alpha$ of $d$, the coefficient of $\binom X\alpha$ in the expansion of $H_\lambda$ in the binomial basis is $\sigma_\lambda(w_\alpha)$, the value of the character $\sigma_\lambda$ of the permutation representation of $S_d$ induced from the trivial representation of the Young subgroup $S_{\lambda_1}\times \dotsb \times S_{\lambda_l}$ (see \cite[Section~2.2]{MR644144} or \cite[Section~2.3]{MR3287258}) at a permutation $w_\alpha$ with cycle type $\alpha$.
\end{lemma}
\begin{proof}
  Consider the set of ordered set partitions of $[d]$ with parts of sizes given by $\lambda$:
  \begin{displaymath}
    X_\lambda = \{(T_1,\dotsc,T_l)\mid [d]=T_1\cup\dotsb \cup T_l \text{ is a set partition, with $|T_i|=\lambda_i$}\}.
  \end{displaymath}
  Let $K[X_\lambda]$ be the permutation representation associated to the action of $S_d$ on $X_\lambda$.
  This representation is isomorphic to the representation of $S_d$ induced from the trivial representation of its Young subgroup $S_{\lambda_1}\times \dotsb \times S_{\lambda_l}$.
  Therefore $\sigma_\lambda$ is the character of $K[X_\lambda]$, and $\sigma_\lambda(w_\alpha)=|X_\lambda^{w_\alpha}|$, the number of fixed points of a permutation $w_\alpha$ in $X_\lambda$.
  Take $(T_1,\dotsc,T_l)\in X_\lambda^{w_\alpha}$.
  Then each $T_i$ is formed by taking a union of cycles of $w_\alpha$.
  Suppose that $b_{ij}$ is the number of $j$ cycles of $w_\alpha$ in $T_i$.
  Then the array $(b_{ij})$ satisfies the constraints:
  \begin{gather}
    \label{eq:brow}
    b_{i1}+2b_{i2}+\dotsb = \lambda_i \text{ for each $i$},\\
    b_{1j}+b_{2j}+\dotsb = a_j \text{ for each $j$},
  \end{gather}
  where $a_j$ is the number of $j$-cycles in $\alpha$.
  Let $B(\lambda;\alpha)$ denote the set of such arrays.
  We have:
  \begin{equation}
    \label{eq:fixed-pt-sum}
    \tr(w_\alpha,K[X_\lambda]) = \sum_{(b_{ij})\in B(\lambda;\alpha)} \prod_j\binom{a_j}{b_{1j}\;b_{2j}\;\dotsb}.
  \end{equation}
  On the other hand, by \eqref{eq:sym-char},
  \begin{displaymath}
    H_\lambda = \sum_{b_{i1}+2b_{i2}+\dotsb=\lambda_i}\prod_{i\geq 1}\prod_{j\geq 1}\mch{X_j}{b_{ij}}.
  \end{displaymath}
  The terms of homogeneous degree $d$ in this product come from the top degree terms in each factor.
  When $|\alpha|=|\lambda|$, such a term has leading coefficient $\prod_j X_j^{a_j}$ if and only if $(b_{ij})\in B(\lambda;\alpha)$.
  Hence the coefficient of $\binom X\alpha$ is the expression on the right hand side of \eqref{eq:fixed-pt-sum} and the lemma follows.
\end{proof}
\begin{proof}
  [Proof of Theorem~\ref{theorem:binomial-coeff}]
  For partitions $\lambda$ and $\mu$ of $d$, let $K_{\mu\lambda}$ denote the number of semistandard Young tableaux of shape $\mu$ and weight $\lambda$.
  Then $K=(K_{\mu\lambda})$ is a unitriangular integer matrix with rows and columns indexed by partitions of $d$.
  We have:
  \begin{align}
    \label{eq:kostka1}
    H_\lambda &= \sum_\mu K_{\mu\lambda}S_\mu,\\
    \label{eq:kostka2}
    \sigma_\lambda & = \sum_\mu K_{\mu\lambda}\chi_\mu.
  \end{align}
  Let $K^{-1}_{\mu\lambda}$ be the entries of the inverse matrix $K^{-1}$.
  Then
  \begin{align*}
    S_\lambda & = \sum_\mu K_{\mu\lambda}^{-1}H_\mu & \text{by \eqref{eq:kostka1}}\\
              & \equiv \sum_\mu K^{-1}_{\mu\lambda}\sum_{\alpha} \sigma_\mu(\alpha)\binom X\alpha & \text{ignoring lower deg. terms (Lemma~\ref{lemma:highest-coeffs})}\\
              & = \sum_{\alpha\vdash d}\sum_\mu K^{-1}_{\mu\lambda}\sigma_\mu(\alpha) \binom X\alpha \\
              & = \sum_\alpha \chi_\lambda(\alpha)\binom X\alpha & \text{by \eqref{eq:kostka2}},
  \end{align*}
  thereby completing the proof of Theorem~\ref{theorem:binomial-coeff}.
\end{proof}
\begin{theorem}
  \label{theorem:weyl-char-poly-gen}
  For every partition $\lambda=(\lambda_1,\dotsc,\lambda_l)$, $S_\lambda$ is the coefficient of $t_1^{\lambda_1}\dotsb t_l^{\lambda_l}$ in
  \begin{displaymath}
    \prod_{i<j}(1-t_j/t_i)\prod_{r=1}^l \prod_{i\geq 1} (1-t_r^i)^{-X_i}.
  \end{displaymath}
\end{theorem}
\begin{proof}
  For every vector $\lambda=(\lambda_1,\dotsc,\lambda_l)$ with non-negative integer coefficients, define:
  \begin{displaymath}
    S_\lambda = \det(H_{\lambda_i-i+j}).
  \end{displaymath}
  When $\lambda$ is a partition this coincides with the character polynomial of the Weyl module $W_\lambda$.
  Then, for every partition $\lambda$, $S_\lambda$ is the coefficient of $t^\lambda$ in $\sum_{\lambda\geq 0} S_\lambda t^\lambda$.
  Here $\lambda\geq 0$ indicates that the sum is over all vectors in $\ZZ_{\geq 0}^l$, and $t^\lambda = t_1^{\lambda_1}\dotsb t_l^{\lambda_l}$.
  Now
  \begin{align*}
    \sum_{\lambda\geq 0}S_\lambda t^\lambda & = \sum_{\lambda\geq 0}\sum_{w\in S_l} \sgn(w)\prod_{r=1}^l H_{\lambda_r-r+w(r)}t_r^{\lambda_r}\\
                                            & = \sum_{w\in S_l} \sgn(w) \prod_{r=1}^l t_r^{r-w(r)} \sum_{\lambda_r\geq 0} H_{\lambda_r-r+w(r)}t_r^{\lambda_r-r+w(r)}\\
                                            & = \sum_{w\in S_l} \sgn(w) \prod_{r=1}^l t_r^{r-w(r)} \sum_{\lambda_r\geq 0} H_{\lambda_r}t_r^{\lambda_r}\\
    & = \prod_{i<j}(1-t_j/t_i)\prod_{r=1}^l\prod_{i\geq 1}(1-t_r^i)^{-X_i}.
  \end{align*}
  Here we have used the convention that $H_d=0$ when $d<0$.
\end{proof}
\subsection{Duality}
\label{sec:duality}
Going through the entries of Table~\ref{table-char_poly}, the reader may have noticed that the coefficients in the expansion of $S_\lambda$ agree up to sign with those of $S_{\lambda'}$ for every partition $\lambda$.
For each vector $\mu=(\mu_1,\dotsc,\mu_m)$ with non-negative integer entries, let $X^\mu = X_1^{\mu_1}\dotsb X_m^{\mu_m}$, and $|\mu|=\mu_1+\dotsb+\mu_m$.
\begin{theorem}
  For every partition $\lambda$, if $S_\lambda = \sum_\mu a^\lambda_\mu X^\mu$, then $S_{\lambda'} = \sum_\mu (-1)^{|\lambda|-|\mu|}a^\lambda_\mu X^\mu$.
\end{theorem}
\begin{proof}
  Let $\tau_d:P\to P$ denote the linear involution defined by $X^\mu\mapsto (-1)^{d-|\mu|}X^\mu$.
  Comparing equations \eqref{eq:H} and \eqref{eq:E} shows that:
  \begin{displaymath}
    \tau_d(H_d)=E_d.
  \end{displaymath}
  It follows that, if $|\mu|=d$, then
  \begin{displaymath}
    \tau_d(H_{\mu_1}\dotsb H_{\mu_m}) = E_{\mu_1}\dotsb E_{\mu_m}.
  \end{displaymath}
  When $\lambda$ is a partition of $d$, then every term in the expansion of the Jacobi--Trudi determinants $\det(H_{\lambda_i+j-i})$ is of the form $H_\mu$ or $E_\mu$ for integer vector $\mu$ with $|\mu|=d$.
  Therefore $\tau_d(\det(H_{\lambda_i+j-i})) = \det(E_{\lambda_i+j-i})$.
  By the Jacobi--Trudi identities \eqref{eq:jt},
  \begin{displaymath}
    \tau_d(S_\lambda) = \tau_d(\det(H_{\lambda_i+j-i})) = \det(E_{\lambda_i+j-i}) = S_{\lambda'},
  \end{displaymath}
  as claimed.
\end{proof}
\section{The Restriction Problem}
\label{sec:relat-with-char}
\subsection{Character Polynomials of Specht Modules}
\label{sec:char-polyn-specht}
Recall that for partitions $\lambda$ and $\mu$, we say that $\lambda-\mu$ is a vertical strip if the Young diagram of $\mu$ is contained inside the Young diagram of $\lambda$, and each row of the Young diagram of $\lambda$ contains at most one box that is not in the Young diagram of $\mu$ \cite[Section~I.1]{MR3443860}.

For any partition $\lambda$, Macdonald \cite[Example~I.7.14(b)]{MR3443860} gave the character polynomials $q_\lambda\in P$ of \eqref{eq:qla} as follows:
\begin{equation}
  \label{eq:q_mu-mac}
  q_\lambda = \sum_{\{\mu\mid \lambda-\mu \text{ is a vertical strip}\}} (-1)^{|\lambda|-|\mu|} \sum_{\alpha \vdash |\mu|} \chi_\mu(\alpha)\binom X\alpha.
\end{equation}

It immediately follows that the leading coefficients of $q_\lambda$ in the binomial basis are the same as those of $S_\lambda$ (see Theorem~\ref{theorem:binomial-coeff}):
\begin{theorem}
  Let $\lambda$ be a partition of a positive integer $d$.
  For every partition $\alpha$ of $d$, the coefficient of $\binom X\alpha$ in the expansion of $q_\lambda$ in the binomial basis of $P$ is $\chi_\lambda(\alpha)$.
\end{theorem}
\begin{corollary}
  \label{corollary:qla-leading}
  The sets:
  \begin{align*}
    \mathbf S & = \{S_\lambda \mid \lambda \text{ is an integer partition}\},\\
    \mathbf q & = \{q_\lambda \mid \lambda \text{ is an integer partition}\}
  \end{align*}
  are bases of $P$.
\end{corollary}
\begin{proof}
  Regard $P$ as a graded algebra where the degree of $X_i$ is $i$ for each $i\geq 1$.
  Let $P_d$ denote the homogeneous elements of degree $d$ in $P$.
  The degree $d$ homogeneous parts of $\binom X\alpha$, as $\alpha$ runs over all partitions of $d$, form a basis of $P_d$.
  Theorem~\ref{theorem:binomial-coeff} and the identity (\ref{eq:q_mu-mac}) imply that the degree $d$ homogeneous parts of $S_\lambda$ and $q_\lambda$ also form such a basis as $\lambda$ runs over all partitions of $d$, since the character table of $S_d$ forms a non-singular matrix.
  Therefore $\mathbf S$ and $\mathbf q$ are bases of $P$.
\end{proof}
\subsection{Stable Restriction Coefficients}
The coefficients in the expansion of the elements of the basis $\mathbf S$ in terms of the basis $\mathbf q$:
\begin{displaymath}
  S_\lambda = \sum_\mu r_{\lambda\mu} q_\mu
\end{displaymath}
are called the stable restriction coefficients.
They determine the decomposition of a Weyl module $W_\lambda(K^n)$ into irreducible representations of $S_n$:
\begin{displaymath}
  \Res^{GL_n(K)}_{S_n}W_\lambda(K^n) = \bigoplus_\mu V_{\mu[n]}^{\oplus r_{\lambda\mu}}.
\end{displaymath}
The following result, which is now immediate, is an algorithm for computing the stable restriction coefficients:
\begin{theorem}
  \label{theorem:stable-res}
  For any partitions $\lambda$ and $\mu$,
  \begin{displaymath}
    r_{\lambda\mu} = \mom{S_\lambda q_\mu}.
  \end{displaymath}
\end{theorem}
The polynomial $S_\lambda$ can be computed using Theorem~\ref{theorem:char-poly-weyl}, $q_\mu$ using \eqref{eq:q_mu-mac}.
After expanding the product in the binomial basis, the moment can be computed using Theorem~\ref{theorem:moment}.
The matrix of the stable restriction coefficients $r_{\lambda\mu}$, as $\lambda$ and $\mu$ run over partitions of $0\leq n\leq 5$ is given by:
\begin{displaymath}
  \left(
    \begin{array}{r|r|rr|rrr|rrrrr|rrrrrrr}
      1 & 0 & 0 & 0 & 0 & 0 & 0 & 0 & 0 & 0 & 0 & 0 & 0 & 0 & 0 & 0 & 0 & 0 & 0 \\
      \hline
      1 & 1 & 0 & 0 & 0 & 0 & 0 & 0 & 0 & 0 & 0 & 0 & 0 & 0 & 0 & 0 & 0 & 0 & 0 \\
      \hline
      2 & 2 & 1 & 0 & 0 & 0 & 0 & 0 & 0 & 0 & 0 & 0 & 0 & 0 & 0 & 0 & 0 & 0 & 0 \\
      0 & 1 & 0 & 1 & 0 & 0 & 0 & 0 & 0 & 0 & 0 & 0 & 0 & 0 & 0 & 0 & 0 & 0 & 0 \\
      \hline
      3 & 4 & 2 & 1 & 1 & 0 & 0 & 0 & 0 & 0 & 0 & 0 & 0 & 0 & 0 & 0 & 0 & 0 & 0 \\
      1 & 3 & 2 & 2 & 0 & 1 & 0 & 0 & 0 & 0 & 0 & 0 & 0 & 0 & 0 & 0 & 0 & 0 & 0 \\
      0 &0 & 0 & 1 & 0 & 0 & 1 & 0 & 0 & 0 & 0 & 0 & 0 & 0 & 0 & 0 & 0 & 0 & 0 \\
      \hline
      5 & 7 & 5 & 2 & 2 & 1 & 0 & 1 & 0 & 0 & 0 & 0 & 0 & 0 & 0 & 0 & 0 & 0 & 0 \\
      2 & 7 & 5 & 6 & 2 & 3 & 1 & 0 & 1 & 0 & 0 & 0 & 0 & 0 & 0 & 0 & 0 & 0 & 0 \\
      2 & 3 & 4 & 1 & 1 & 2 & 0 & 0 & 0 & 1 & 0 & 0 & 0 & 0 & 0 & 0 & 0 & 0 & 0 \\
      0 & 1 & 1 & 3 & 0 & 2 & 2 & 0 & 0 & 0 & 1 & 0 & 0 & 0 & 0 & 0 & 0 & 0 & 0 \\
      0 & 0 & 0 & 0 & 0 & 0 & 1 & 0 & 0 & 0 & 0 & 1 & 0 & 0 & 0 & 0 & 0 & 0 & 0 \\
      \hline
      7 & 12 & 9 & 5 & 5 & 3 & 0 & 2 & 1 & 0 & 0 & 0 & 1 & 0 & 0 & 0 & 0 & 0 & 0 \\
      5 & 14 & 13 & 12 & 6 & 9 & 3 & 2 & 3 & 1 & 1 & 0 & 0 & 1 & 0 & 0 & 0 & 0 & 0 \\
      4 & 10 & 11 & 8 & 6 & 8 & 2 & 1 & 3 & 2 & 1 & 0 & 0 & 0 & 1 & 0 & 0 & 0 & 0 \\
      0 & 3 & 4 & 8 & 1 & 7 & 6 & 0 & 2 & 1 & 3 & 1 & 0 & 0 & 0 & 1 & 0 & 0 & 0 \\
      1 & 3 & 4 & 3 & 2 & 5 & 1 & 0 & 1 & 2 & 2 & 0 & 0 & 0 & 0 & 0 & 1 & 0 & 0 \\
      0 & 0 & 0 & 1 & 0 & 1 & 3 & 0 & 0 & 0 & 2 & 2 & 0 & 0 & 0 & 0 & 0 & 1 & 0 \\
      0 & 0 & 0 & 0 & 0 & 0 & 0 & 0 & 0 & 0 & 0 & 1 & 0 & 0 & 0 & 0 & 0 & 0 & 1
    \end{array}
  \right)
\end{displaymath}
The blocks demarcate the partitions of each integer $n$, and within each block, the partitions of $n$ are enumerated in reverse lexicographic order.
\subsection{Relation to Symmetric Functions}
\label{sec:oz}
Let $\Lambda$ denote the ring of symmetric functions (as in \cite[Section~I.2]{MR3443860}).
Macdonald \cite[Example~I.7.13]{MR3443860} constructed an isomorphism $\phi:\Lambda\to P$ taking the Schur function $s_\lambda$ to the character polynomial $q_\lambda$.
In this section we study a different isomorphism $\Phi:\Lambda\to P$, due to Orellana and Zabrocki \cite{2016arXiv160506672O}, which takes $s_\lambda$ to $S_\lambda$.
Under this isomorphism $q_\lambda$ is the image of $\tilde s_\lambda$ the \emph{Specht symmetric functions} of \cite{assaf-speyer,2016arXiv160506672O}.

Following \cite[Proposition~12]{2016arXiv160506672O}, define an algebra homomorphism\linebreak $\Phi:\Lambda\to P$ by:
\begin{equation}
  \Phi:p_k\mapsto \sum_{d|k} dX_d.
\end{equation}
For each $k>0$, define
\begin{displaymath}
  \Xi_k = 1, e^{2\pi i/k}, e^{4\pi i/k},\dotsc e^{2(k-1)\pi i/k}, 
\end{displaymath}
and for an integer partition $\mu=(\mu_1,\dotsc,\mu_m)$,
\begin{displaymath}
  \Xi_\mu = \Xi_{\mu_1},\dotsc,\Xi_{\mu_m}.
\end{displaymath}

Let $R_n$ denote the space of $K$-valued class functions on $S_n$.
For every $n\geq 0$ there is a map $\ev_\Lambda^n:\Lambda\to R_n$ defined by:
\begin{displaymath}
  \ev_\Lambda^n f(w) = f(\Xi_\mu),
\end{displaymath}
where $\mu$ is the cycle type of $w$.
In other words, the symmetric function is evaluated on $|\mu|$ variables, whose values are given by the list $\Xi_\mu$, the remaining variables being set to $0$.
With this definition, $\ev^n_\Lambda(s_\lambda)$ is the character of $\Res^{GL_n(K)}_{S_n} W_\lambda(K^n)$.

For each $q\in P$ consider the function $\ev_P^n(q)\in R_n$ given by:
\begin{displaymath}
  \ev_P^n(q)(w) = q(X_1(w),X_2(w),\dotsc).
\end{displaymath}
This defines a ring homomorphism $\ev_P^n:P\to R_n$.

Observe that $\oplus_{n=1}^\infty \ev_P^n:P\to \oplus_{n=1}^\infty R_n$ is injective, for if $\ev_P^n(q)\equiv 0$ for all $n$, then $q$ vanishes whenever $X_1,X_2,\dotsc$ take non-negative integer values, and hence $q$ must be identically $0$.

\begin{theorem}
  \label{theorem:Phi}
  The algebra homomorphism $\Phi$ is the unique $K$-linear map $\Lambda\to P$ such that the  diagram
  \begin{equation}
    \label{eq:oznpps}
    \xymatrix{
      \Lambda\ar[rr]^\Phi\ar[dr]_{\ev_\Lambda^n}& & P\ar[dl]^{\ev_P^n}\\
      & R_n &
    }
  \end{equation}
  commutes for every $n\geq 1$.
\end{theorem}
\begin{proof}
  From the definition of $\ev^n_\Lambda$,
  \begin{displaymath}
    \ev_\Lambda^n(p_k)(w) = \sum_d X_d(w) \sum_{j=0}^{d-1}(e^{2\pi i/d})^{jk}.
  \end{displaymath}
  Now observe that
  \begin{displaymath}
    \sum_{j=0}^{d-1} (e^{2\pi i/d})^{jk} =
    \begin{cases}
      d & \text{if } d|k,\\
      0 & \text{otherwise}.
    \end{cases}
  \end{displaymath}
  It follows that
  \begin{displaymath}
    \ev_\Lambda^n(p_k)(w) = \sum_{d|k} dX_d(w) = \ev_P^n(\Phi(p_k)).
  \end{displaymath}
  Since $\oplus_n \ev_P^n:P\to \oplus R_n$ is injective, $\Phi(p_k)$ is completely determined by the commutativity of \eqref{eq:oznpps}.
  Since the polynomials $\{p_k\}_{k\geq 1}$ generate $\Lambda$, $\Phi$ is completely determined by its values on $p_k$.
\end{proof}
\begin{lemma}
  The homomorphism $\Phi:\Lambda\to P$ is an isomorphism of rings.
\end{lemma}
\begin{proof}
  The inverse of $\Phi$ is obtained using the M\"obius inversion formula:
  \begin{displaymath}
    X_k\mapsto \frac 1d \sum_{d|k}\mu(k/d) p_d,
  \end{displaymath}
  where $\mu$ denotes the number-theoretic M\"obius function (see, e.g., \cite[Section~3.1.1]{rota}).
\end{proof}
\begin{theorem}
  For every partition $\lambda$, we have:
  \begin{align}
    \Phi(s_\lambda) & = S_\lambda,\\
    \label{eq:oz}
    \Phi(\tilde s_\lambda) & = q_\lambda. 
  \end{align}
\end{theorem}
\begin{proof}
  This follows immediately from Theorem~\ref{theorem:Phi}.
\end{proof}
\begin{remark}
  The second identity \eqref{eq:oz} is \cite[Prop.~12]{2016arXiv160506672O}.
\end{remark}
\section{Moment Generating Functions}
\label{sec:moment-gener-funct}

\subsection{The Main Generating Function}
\label{sec:master-generating-function}
In this and the following subsections, we shall frequently use the following elementary identities:
\begin{gather}
  \tag A
  \label{eq:exp}
  \exp(t/i) = \sum_{b\geq 0} \frac 1{i^bb!} t^b,\\
  \tag B
  \label{eq:log}
  \log \frac 1{1-t} = \sum_{i=1}^\infty t^i/i.
\end{gather}
We shall use $\alpha=1^{a_1}2^{a_2}\dotsb$, $\beta = 1^{b_1}2^{b_2}\dotsb$, $\gamma = 1^{c_1}2^{c_2}\dotsb$.
Let $\Par$ denote the set of all integer partitions.
We use the notation $\lambda=(\lambda_1,\dotsc,\lambda_l)$, $\mu=(\mu_1,\dotsc,\mu_m)$.
Also, $t^\lambda = t_1^{\lambda_1}\dotsb t_l^{\lambda_l}$ and $u^\mu = u_1^{\mu_1}\dotsb u_m^{\mu_m}$.
We shall interpret $\lambda\geq 0$ as $\lambda_i\geq 0$ for $i=1,\dotsc,l$ and $\mu\geq 0$ as $\mu_i\geq 0$ for $i=1,\dotsc,m$.

We use the notation $R\sqsubset [l]$ to signify that $R$ is a multiset with elements drawn from $[l]$.
We write $t^R$ for the monomial where $t_i$ is raised to the multiplicity of $i$ in $R$.
Similarly, for any $S\subset [m]$, we write $u^S=\prod_{i\in S}u_i$.

\begin{theorem}
  \label{theorem:hlaemu}
  We have:
  \begin{displaymath}
    \sum_{n\geq 0,\lambda\geq 0,\mu\geq 0} \mom{H_\lambda E_\mu}_n t^\lambda u^\mu v^n = \prod_{R\sqsubset [l]}\frac{\prod_{S\subset [m],\;|S|\text{ odd}}(1+u^St^Rv)}{\prod_{S\subset [m],\;|S|\text{ even}}(1-u^St^Rv)}.
  \end{displaymath}
\end{theorem}
\begin{proof}
  Using \eqref{eq:H} and \eqref{eq:E}, we have:
  \begin{displaymath}
    \sum_{\lambda\geq 0,\mu\geq 0} H_\lambda E_\mu t^\lambda u^\mu = \prod_{r=1}^l\prod_{s=1}^m \prod_{i\geq 1} \left(\frac{1-(-u_s)^i}{1-t_r^i}\right)^{X_i}.
  \end{displaymath}
  Now proceeding as in the proof of Theorem~\ref{theorem:moment},
  \begin{align*}
    \sum_{n\geq 0,\lambda\geq 0,\mu\geq 0} \mom{H_\lambda E_\mu}_n t^\lambda u^\mu v^n & = \prod_{i\geq 1}\sum_{b_i\geq 0}\frac{v^{ib_i}}{i^{b_i}b_i!}\prod_{r=1}^l\prod_{s=1}^m \left(\frac{1-(-u_s)^i}{1-t^i_r}\right)^{b_i}\\
                                                                                       & \overset{\eqref{eq:exp}}{=} \prod_{i\geq 1}\exp\left(\frac{v^i}{i}\prod_{r=1}^l\prod_{s=1}^m\left[\frac{1-(-u_s)^i}{1-t^i_r}\right]\right) \\
                                                                                       & = \exp\left(\sum_{i\geq 1}\sum_{R\sqsubset[l]}\sum_{S\subset[m]}(-1)^{|S|} \frac{(t^R(-1)^{|S|}u^Sv)^i}i\right)\\
                                                                                       & \overset{~\eqref{eq:log}}{=} \prod_{R\sqsubset[l]}\prod_{S\subset[m]} \left(1-(-1)^{|S|}t^Ru^Sv\right)^{(-1)^{|S|+1}},
  \end{align*}
  which is equivalent to the desired expression.
\end{proof}
\begin{corollary}
  \label{corollary:hlamu}
  We have:
  \begin{displaymath}
    \sum_{\lambda\geq 0,\mu\geq 0} \mom{H_\lambda E_\mu} t^\lambda u^\mu = \prod_{R,S}(1-(-1)^{|S|}u^St^R)^{(-1)^{|S|+1}},
  \end{displaymath}
  where the product is over $R\sqsubset[l]$, $S\subset[m]$, with at least one of $R$ and $S$ non-empty.
\end{corollary}
\begin{corollary}
  \label{corollary:weyl-mom}
  For every partition $\lambda$, $\mom{W_\lambda}_n$ is the coefficient of $t^\lambda v^n$ in
  \begin{displaymath}
    \prod_{i<j}(1-t_j/t_i) \prod_{R\sqsubset[l]}(1-t^Rv)^{-1}.
  \end{displaymath}
\end{corollary}
\begin{proof}
  From Theorem~\ref{theorem:hlaemu} we get:
  \begin{equation}
    \label{eq:hmomgen}
    \sum_{\lambda\geq 0,n\geq 0} \mom{H_\lambda}_n t^\lambda v^n = \prod_{R\sqsubset [l]} (1-t^Rv)^{-1}.
  \end{equation}
  Using this, the corollary can be deduced from Theorem~\ref{theorem:weyl-char-poly-gen} by taking moments. 
\end{proof}
\subsection{$S_n$-invariant Vectors}
\label{sec:s_n-invar-vect}
For a representation $V_n$ of $S_n$, let $V_n^{S_n}$ denote the subspace of $S_n$-invariant vectors.
If a family $\{V_n\}$ of representations has polynomial character $q\in P$, then
\begin{displaymath}
  \mom q_n = \dim(V_n^{S_n}) \text{ for all $n\geq 0$}.
\end{displaymath}
Therefore, for any partition $\lambda$, $\dim W_\lambda(K^n)^{S_n}=\mom{S_\lambda}_n$.
In particular, $W_\lambda(K^n)$ has a non-zero $S_n$-invariant vector if and only if $\mom{S_\lambda}_n\neq 0$.
\begin{theorem}
  \label{theorem:monotonicity}
  For every positive integer $n$ and every partition $\lambda$ with at most $n$ parts,
  \begin{displaymath}
    \dim W_\lambda(K^n)^{S_n} \leq \dim W_\lambda(K^{n+1})^{S_{n+1}}.
  \end{displaymath}
\end{theorem}
\begin{proof}
  We use Littlewood's plethystic formula \cite[Theorem XI]{MR95209} (see also, \cite[Theorem~2.6]{poly-ind}) for restriction coefficients.
  This formula asserts that, for every partition $\lambda$ with at most $n$ parts, and every partition $\mu$ of $n$, the multiplicity of the Specht module $V_\mu$ in $\Res^{GL_n(K)}_{S_n} W_\lambda(K^n)$ is given by $(s_\lambda, s_\mu[H])$.
  Here $(-,-)$ denotes the Hall inner product on symmetric functions, and $s_\mu[H]$ denotes the plethystic substitution of $H$ into $s_\mu$ (for definitions, see \cite[Section~1]{MR2576382}).
  Taking $\mu = (n)$ in Littlewood's formula gives
  \begin{displaymath}
    \dim W_\lambda(K^n)^{S_n} = (s_\lambda, h_n[H]).
  \end{displaymath}
  Recall \cite[Eq.~1.8]{MR2576382} that
  \begin{displaymath}
    h_{n+1}[H] - h_n[H] = h_{n+1}[H-1],
  \end{displaymath}
  so that
  \begin{displaymath}
    \dim W_\lambda(K^{n+1})^{S_{n+1}} - \dim W_\lambda(K^n)^{S_n} = (s_\lambda,h_{n+1}[H-1]) \geq 0.
  \end{displaymath}
  The inequality above holds because the plethystic substitution of a Schur-positive symmetric function into another is Schur-positive.
\end{proof}
\begin{definition}
  [Vector Partitions]
  Let $\vv\in \ZZ_{\geq 0}^l$.
  A \emph{vector partition} of $\vv$ is an unordered collection $\vv_1,\dotsc,\vv_n$ of non-zero vectors in $\ZZ^l_{\geq 0}$ such that
  \begin{displaymath}
    \vv = \vv_1+\dotsb + \vv_n.
  \end{displaymath}
  Let $p_n(\vv)$ (resp. $p_{\leq n}(\vv)$) denote the number of vector partitions of $\vv$ with exactly (resp. at most) $n$ parts.
\end{definition}
\begin{theorem}
  \label{theorem:weyl-mom}
  For every partition $\lambda=(\lambda_1,\dotsc,\lambda_l)$,
  \begin{displaymath}
    \dim W_\lambda(K^n)^{S_n} = \sum_{w\in S_l} \sgn(w)p_{\leq n}(\lambda_1-1+w(1),\dotsc,\lambda_l-l+w(l)).
  \end{displaymath}
\end{theorem}
\begin{proof}
  The coefficient of $t^\lambda v^n$ in the right hand side of \eqref{eq:hmomgen} is $p_{\leq n}(\lambda)$.
  Therefore,
  \begin{equation}
    \label{eq:hmom}
    \mom{H_\lambda}_n = p_{\leq n}(\lambda) \text{ for every $\lambda\in \ZZ_{\geq 0}^l$}.
  \end{equation}
  By the Jacobi--Trudi identity \eqref{eq:jt},
  \begin{displaymath}
    S_\lambda = \sum_{w\in S_l} H_{\lambda_1-1+w(1),\dotsc,\lambda_l-l+w(l)},
  \end{displaymath}
  so by \eqref{eq:hmom},
  \begin{displaymath}
    \mom{S_\lambda}_n =  \sum_{w\in S_l} \sgn(w)p_{\leq n}(\lambda_1-1+w(1),\dotsc,\lambda_l-l+w(l)),
  \end{displaymath}
  as claimed.
\end{proof}
\begin{remark}
  In general, we do not know of a combinatorial proof of the non-negativity of $\sum_{w\in S_l} \sgn(w)p_{\leq n}(\lambda_1-1+w(1),\dotsc,\lambda_l-l+w(l))$, which follows from Theorem~\ref{theorem:weyl-mom}.
  When $l=2$, this is the main result of Kim and Hahn~\cite{Kim1997}, who refer to it as a conjecture of Landman, Brown and Portier \cite{landman1992partitions}.
\end{remark}
  The problem of characterizing those partitions $\lambda$ for which $W_\lambda(K^n)$ has a non-zero $S_n$-invariant vector for large $n$ appears to be quite hard.
  The following result solves this problem for partition with two rows, two columns, and for hook-partitions.
  \begin{theorem}
    \label{theorem:2row2colhook}
    Let $\lambda$ be a partition.
    \renewcommand{\theenumi}{\thetheorem.\arabic{enumi}}
  \begin{enumerate}
  \item \label{item:1} If $\lambda=(\lambda_1,\lambda_2)$, then $\mom{S_\lambda}>0$ unless $\lambda=(1,1)$.
  \item \label{item:2} If $\lambda'=(\lambda_1,\lambda_2)$, then $\mom{S_\lambda}>0$ if and only if $\lambda_1=\lambda_2$ (in which case $\mom{S_\lambda}=2$) or $\lambda_1=\lambda_2+1$ (in which case $\mom{S_\lambda}=1$).
  \item \label{item:3} If $\lambda = (a+1,1^b)$, then $\mom{S_\lambda}>0$ if and only if $a\geq \binom{b+1}2$.
  \end{enumerate}
\end{theorem}
\begin{proof}
  [Proof of (\ref{item:1})]
  By Theorem~\ref{theorem:weyl-mom} we need to show that, for every $\lambda_1\geq \lambda_2\geq 1$,
  \begin{displaymath}
    p_{\leq n}(\lambda_1,\lambda_2)>p_{\leq n}(\lambda_1+1,\lambda_2-1)
  \end{displaymath}
  for sufficiently large $n$, unless $\lambda_1=\lambda_2=1$.
  From the main result of Kim and Hahn \cite{Kim1997} (the result on the last line of the first page), it follows that
  \begin{displaymath}
    p_n(\lambda_1,\lambda_2)\geq p_n(\lambda_1+1,\lambda_2-1) \text{ for all $n\geq 1$.}
  \end{displaymath}
  Therefore, it suffices to prove that $p_n(\lambda_1,\lambda_2)> p_n(\lambda_1+1,\lambda_2-1)$ for at least one value of $n$.
  When $k\geq l\geq 1$ are such that at least one of $k$ and $l$ is even, $p_2(k,l)>p_2(k+1,l-1)$.
  When both $k$ and $l$ are odd and $(k,l)\neq (1,1)$, $p_3(k,l)>p_3(k+1,l-1)$.
  These inequalities will be proved in Lemmas~\ref{lemma:p2} and~\ref{lemma:p3} below.
\end{proof}
\begin{lemma}
  \label{lemma:p2p3}
  For all $k,l\geq 0$,
  \begin{align}
    \label{eq:p2}
    p_2(k,l) & =
    \begin{cases}
      \frac{(k+1)(l+1)-1}2 & \text{if both $k$ and $l$ are even,}\\
      \frac{(k+1)(l+1)}2-1 & \text{otherwise},
    \end{cases}
    \\
    \label{eq:p3}
    p_3(k,l) & = \frac 16(A+3B+2C),
  \end{align}
  where
  \begin{align*}
    A & = \binom{k+2}2\binom{l+2}2-3(k+1)(l+1)+3,\\
    B & =
        \begin{cases}
          (k/2+1)(l/2+1)-2 &\text{if $k$ and $l$ are even,}\\
          (k+1)(l+2)/4-1 &\text{if $k$ is odd and $l$ is even,}\\
          (k+2)(l+1)/4-1 &\text{if $k$ is even and $l$ is odd,}\\
          (k+1)(l+1)/4 - 1 &\text{otherwise},
        \end{cases}
    \\
    C & =
        \begin{cases}
          1 &\text{if $k$ and $l$ are divisible by $3$,}\\
          0 &\text{otherwise.}
        \end{cases}
  \end{align*}
\end{lemma}
\begin{proof}
  Consider the set of all ordered triples $((k_1,l_1),(k_2,l_2),(k_3,l_3))$ such that $\sum_i (k_i,l_i)=(k,l)$ and no $(k_i,l_i)=(0,0)$.
  The group $S_3$ acts by permutation on the set of all such triples, and the number of orbits if $p_3(k,l)$.
  The quantities $A$, $B$, and $C$ in Lemma~\ref{lemma:p2p3} are the number of such triples that are fixed by permutations in $S_3$ of cycle types $(1,1,1)$, $(2,1)$, and $(3)$, respectively.
  The formula for $p_3(k,l)$ then follows from Burnside's lemma.
  The formula for $p_2(k,l)$ is obtained in a similar fashion.
\end{proof}
\begin{lemma}
  \label{lemma:p2}
  For all integers $k\geq l\geq 1$ such that at least one of $k$ and $l$ is even,
  \begin{displaymath}
    p_2(k,l)>p_2(k+1,l-1).
  \end{displaymath}
\end{lemma}
\begin{proof}
    By Lemma~\ref{lemma:p2p3}, we also have:
  \begin{displaymath}
    p_2(k+1,l-1) =
    \begin{cases}
      \frac{(k+2)l-1}2 &\text{if $k$ and $l$ are odd,}\\
      \frac{(k+2)l}2 -1 &\text{otherwise}.
    \end{cases}
  \end{displaymath}
  Thus, if $k$ and $l$ are both even and $k\geq l$, then
  \begin{align*}
    p_2(k,l)-p_2(k+1,l-1)&=\frac{(k+1)(l+1)-1}2-\left(\frac{(k+2)l}2-1\right)\\
                         &=\frac{k-l}2+1>0.
  \end{align*}
  If one of $k$ and $l$ is even and the other is odd, then
  \begin{align*}
    p_2(k,l)-p_2(k+1,l-1)&=\frac{(k+1)(l+1)}2-1-\left(\frac{(k+2)l}2-1\right)\\
                         &=\frac{k-l+1}2>0,
  \end{align*}
  thereby completing the proof of Lemma~\ref{lemma:p2}.
\end{proof}
\begin{lemma}
  \label{lemma:p3}
  If $k\geq l\geq 1$, both $k$ and $l$ are odd, and $(k,l)\neq (1,1)$, then $p_3(k,l)>p_3(k+1,l-1)$.
\end{lemma}
\begin{proof}
  When $k$ and $l$ are odd, Lemma~\ref{lemma:p2p3} gives:
  \begin{multline*}
    12(p_3(k,l)-p_3(k+1,l-1)) \\ =
    \begin{cases}
      (kl+2l)(k-l) + k(k-3) + 3l + 4 & \text{if $3|k$ and $3|l$},\\
      (kl+2l)(k-l) + k(k-3) + 3l & \text{otherwise}.
    \end{cases}
  \end{multline*}
  This is clearly positive for all $k\geq l$ such that $k\geq 3$ and $l\geq 1$.
\end{proof}
\begin{proof}
  [Proof of (\ref{item:2})]
  By the second Jacobi--Trudi identity,
  \begin{equation}
    \label{eq:wlap}
    W_{\lambda'} = E_{\lambda_1}E_{\lambda_2} - E_{\lambda+1}E_{\lambda_2-1}.
  \end{equation}
  Taking $l=0$ and $m=2$ Corollary~\ref{corollary:hlamu} gives:
  \begin{displaymath}
    \sum_{\lambda_1,\lambda_2\geq 0} \mom{E_{\lambda_1}E_{\lambda_2}}u_1^{\lambda_1}u_2^{\lambda_2} = \frac{(1+u_1)(1+u_2)}{(1-u_1u_2)}.
  \end{displaymath}
  Therefore the coefficient of $u_1^{\lambda_1}u_2^{\lambda_2}$ is the number $p'(\lambda_1,\lambda_2)$ of ways of writing $(\lambda_1,\lambda_2)$ as a sum of vectors of the form $(1,1)$, $(1,0)$ and $(0,1)$, where the vectors $(0,1)$ and $(1,0)$ are used at most once.
  Clearly
  \begin{displaymath}
    p'(\lambda_1,\lambda_2) =
    \begin{cases}
      2 & \text{if $\lambda_1=\lambda_2 \geq 1$},\\
      1 & \text{if $|\lambda_1-\lambda_2|=1$},\\
      0 & \text{otherwise}.
    \end{cases}
  \end{displaymath}
  By \eqref{eq:wlap}, for any partition $\lambda=(\lambda_1,\lambda_2)$ with two parts,
  \begin{displaymath}
    \mom{W_{\lambda'}} = p'(\lambda_1,\lambda_2)-p'(\lambda_1+1,\lambda_2-1)=
    \begin{cases}
      2 & \text{if $\lambda_1\geq \lambda_2 \geq 1$},\\
      1 & \text{if $\lambda_1-\lambda_2=1$},\\
      0 & \text{otherwise},
    \end{cases}
  \end{displaymath}
  as claimed.
\end{proof}
\begin{proof}
  [Proof of (\ref{item:3})]
  Using Pieri's rule, we have:
  \begin{displaymath}
    h_ke_l=s_{(k-1|l)}+s_{(k|l-1)},
  \end{displaymath}
  whence
  \begin{displaymath}
    s_{(a|b)}=h_{a+1}e_b-h_{a+2}e_{b-1} + \dotsb + (-1)^bh_{a+b+1}e_0.
  \end{displaymath}
  It follows that
  \begin{equation}
    \label{eq:mom-wab}
    \mom{W_{(a|b)}} = \sum_{j=0}^b (-1)^j\mom{H_{a+j+1}E_{b-j}}.
  \end{equation}

  Taking $l=m=1$ in Corollary~\ref{corollary:hlamu} gives:
  \begin{equation}
    \label{eq:he}
    \sum_{a,b\geq 0} \mom{H_iE_j}t^iu^j = \frac{\prod_{k=0}^\infty (1+t^ku)}{\prod_{k=1}^\infty (1-t^k)}.
  \end{equation}
  The coefficient of $t^iu^j$ in the above expression is the number $\tilde p(i,j)$ of ways of writing the vector $(i,j)$ as a sum of vectors of the form $(a,0)$ where $a>0$, and $(a,1)$ where $a\geq 0$, and vectors of the form $(a,1)$ are used at most once.
  If $\tilde p(i,j)>0$, then $j$ distinct vectors of the form $(a,1)$ are used, so that $i\geq \binom j2$.
  Therefore
  \begin{equation}
    \label{eq:zero-region}
    \tilde p(i,j) = 0 \text{ for all $i<\binom j2$}.
  \end{equation}

  If $a< \binom{b+1}2$, then $a-j<\binom{b+j+1}2$ for all $j\geq 0$.
  By \eqref{eq:zero-region} $\mom{H_{a-j}E_{b+j+1}}=0$ for all $j\geq 0$.
  This allows us to extend the index of summation in the right hand side of \eqref{eq:mom-wab} without changing the sum:
  \begin{multline*}
    \mom{W_{(a|b)}} = \sum_{j=-a-1}^b (-1)^j\mom{H_{a+j+1}E_{b-j}}\\ = \sum_{k+l=a+b+1}\mom{H_kE_l} = \mom{W_{(0|a+b+1)}}.
  \end{multline*}
  Taking $l=0$ and $m=1$ in Corollary~\ref{corollary:hlamu} can be used to show that $\mom{E_k}=0$ for all $k>1$, so $\mom{W_{(0|a+b+1)}}=\mom{E_{a+b+2}}=0$ for all $a,b\geq 0$.
  Therefore, $\mom{W_{(a|b)}}=0$ for $a<\binom{b+1}2$.

  Conversely, suppose $a\geq \binom{b+1}2$.
  By Theorem~\ref{theorem:monotonicity}, it suffices to show that $W_\lambda(K^n)$ contains a non-zero $S_n$-invariant vector for some positive integer $n$.
  We shall show that $W_\lambda(K^{b+1})$ contains a non-zero $S_{b+1}$-invariant vector.
  The hook partition $\lambda=(a+1,1^b)$ dominates the partition $\mu=(a-\binom b2+1,b,b-1,\dotsc,2,1)$, which has $b+1$ distinct parts.
  Therefore $W_{(a|b)}(K^{a+b+1})$ contains a non-zero vector $v$ with weight $\mu$.
  For each $w\in S_n$ let $v_w=\rho_{(a|b)}(w)v$.
  Then $v_w$ lies in the weight space of $w\cdot \mu$.
  Hence the vectors $\{v_w\mid w\in S_n\}$ are linearly independent, and generate a representation that is isomorphic to the regular representation of $S_n$.
  In particular, the trivial representation is contained in $W_{(a|b)}(K^{a+b+1})$.
\end{proof}
\section{Strict Polynomial Functors and $\FI$-modules}
\label{sec:strict-polyn-funct-1}
In this section we may take $K$ to be any field (not necessarily of characteristic zero).
Friedlander and Suslin \cite{MR1427618} introduced strict polynomial functors to unify homogeneous polynomial representations of $GL_n(K)$ of degree $d$ across all $n$.
Later, Church, Ellenberg and Farb \cite{MR3357185} introduced $\FI$-modules to unify representations of $S_n$ across all $n$.
In this section, we lift the restriction functor $\Res^{GL_n(K)}_{S_n}$ to a functor from the category of strict polynomial functors to the category of $\FI$-modules.
\subsection{Strict Polynomial Functors}
\label{sec:strict-polyn-funct}
The Schur category (also known as the divided power category, see \cite{vanderKallen2015}) $\mathbf\Gamma^d$ is the category whose objects are finite dimensional vector spaces over $K$.
Given objects $V$ and $W$,
\begin{displaymath}
  \Hom_{\mathbf\Gamma^d}(V,W) = \Hom_{S_d}(V^{\otimes d},W^{\otimes d}).
\end{displaymath}
The category of strict polynomial functors is the category $\Rep\mathbf\Gamma^d$ whose objects are $K$-linear covariant functors from $\mathbf\Gamma^d$ to the category of $K$-vector spaces, and morphisms are natural transformations between functors (see \cite[Section~2]{MR1427618}).
When $K$ has characteristic zero, $\Rep \mathbf\Gamma^d$ is semisimple, and its simple objects are functors known as \emph{Weyl functors} (see \cite{weyl-schur}).
For each partition $\lambda$ of $d$ let $W_\lambda$ denote the Weyl functor corresponding to $\lambda$.

Let $\Rep^d GL_n(K)$ denote the category of homogeneous polynomial representations of $GL_n(K)$ of degree $d$.
Define a functor $\ev_n:\Rep\mathbf\Gamma^d\to \Rep^d GL_n(K)$ as follows: for each strict polynomial functor $F:\mathbf\Gamma^d\to \Vec$ define $\ev_n(F)=F(K^n)$.
Let $T\in GL_n(K)$ act on $F(K^n)$ by $F(T^{\otimes d})$.
This makes $F(K^n)$ a representation of $GL_n(K)$ which turns out to be a homogeneous polynomial representation of degree $d$.
For each $n\geq 1$ and $\lambda\in P(n,d)$, $\ev_n(W_\lambda)= W_\lambda(K^n)$ is the irreducible polynomial representation of $GL_n(K)$ corresponding to $\lambda$ (consistent with the notation of Section~\ref{sec:introduction}).

\subsection{$\FI$-modules}
\label{sec:fi-modules}
The category $\FI$ is the one that has finite sets as objects, and injective functions as morphisms.
The category of $\FI$-modules is the category $\fimod$ whose objects are covariant functors from $\FI$ to the category of $K$-vector spaces, and morphisms are natural transformations of functors.
Let $\Rep S_n$ denote the category of representations of $S_n$ over $K$.
The evaluation functor $\ev_n:\fimod\to \Rep S_n$ is defined by setting $\ev_n(V)=V([n])$, where $[n]=\{1,\dotsc,n\}$.

For each partition $\lambda=(\lambda_1,\dotsc,\lambda_l)$ there exists an $\FI$-module $V(\lambda)$ (see \cite[Proposition~3.4.1]{MR3357185}) such that, for every $n\geq |\lambda|+\lambda_1$, we have:
\begin{displaymath}
  \ev_n(V(\lambda)) = V_{\lambda[n]}.
\end{displaymath}
\subsection{The Restriction Functor}
\label{sec:restriction-functor}
For each object $A$ of $\FI$ let $\F[A]$ be the vector space of all functions $A\to K$.
Given an injective function $i:A\to B$, define $\F(i):\F[A]\to \F[B]$ by:
\begin{displaymath}
  \F(i)(f)(b) =
  \begin{cases}
    f(a) &\text{if there exists $a\in A$ such that $i(a)=b$},\\
    0 & \text{otherwise},
  \end{cases}
\end{displaymath}
for all $f\in \F[A]$.
Then $\F:\FI\to \mathbf\Gamma^1$ is a functor.
For every positive integer $d$, define $\F^d:\FI\to\mathbf\Gamma^d$ by
\begin{displaymath}
  \F^d(A) = \F[A], \quad \F^d(i) = \F(i)^{\otimes d}.
\end{displaymath}
The restriction functor
\begin{displaymath}
  \Res^d:\Rep\mathbf\Gamma^d\to \fimod
\end{displaymath}
is defined by:
\begin{displaymath}
  \Res^d F = F\circ \F^d \text{ for every object $F$ of $\Rep\mathbf\Gamma^d$.}
\end{displaymath}
\begin{theorem}
  \label{theorem:Res}
  The diagram of functors
  \begin{displaymath}
    \xymatrix{
      \Rep \mathbf\Gamma^d \ar[rr]^{\Res^d} \ar[d]_{\ev_n}&& \fimod\ar[d]^{\ev_n}\\
      \Rep^d GL_n(K) \ar[rr]_{\Res^{GL_n(K)}_{S_n}} && \Rep S_n
      }
  \end{displaymath}
  commutes, in the sense that $\Res^{GL_n(K)}_{S_n}\circ \ev_n$ is naturally isomorphic to $\ev_n\circ \Res^d$.
\end{theorem}
\begin{proof}
  Given $F\in \Rep\mathbf\Gamma^d$, $\ev_n(K)$ is $F(K^n)$.
  Given $w\in S_n$, let $T_w\in GL_n(K)$ denote the linear map the takes the $i$th coordinate vector $e_i$ of $K^n$ to $e_{w(i)}$.
  An element $w\in S_n$ acts on $F(K^n)$ via $F(T_w^{\otimes d})$.
  On the other hand,
  \begin{align*}
    \ev_n\circ \Res^d(F)&=F\circ \F^d([n]).
  \end{align*}
  An element $w\in S_n$ acts on $F(\F^d([n]))$ by $F(\F^d(w))=F(\F(w)^{\otimes d})$.
  These two actions of $S_n$ coincide under the isomorphism $K^n\to \F^d([n])$ given by $e_i\mapsto \delta_i$, where $\delta_i$ is the Kronecker delta function on $[n]$ supported at $i$.
  Thus we get an isomorphism $\Res^{GL_n(K)}_{S_n}\circ\ev_n\to\ev_n\circ \Res^d(F)$ of representations of $S_n$.
  The naturality of this isomorphism follows tautologically from unwinding the definitions.
\end{proof}
Church, Ellenberg and Farb introduced the notion of finite generation of $\FI$-modules \cite[Definition~1.2]{MR3357185}.
In characteristic zero, finitely generated $\FI$-modules have eventually polynomial character \cite[Theorem~1.5]{MR3357185}, and in general characteristic, have eventually polynomial dimension \cite[Theorem~B]{CEFN}.
Therefore, the stability of restriction coefficients (discussed in Section~\ref{theorem:stable-res}) is also a consequence of the following theorem.
\begin{theorem}
  \label{theorem:isFI}
  For every finitely generated strict polynomial functor $F$ of degree $d$, the $\FI$-module $\Res^d F$ is finitely generated in degree $d$.
\end{theorem}
\begin{proof}
  Let $\Gamma^d$ denote the strict polynomial functor $\Gamma^d(V)=(V^{\otimes d})^{S_d}$, the subspace of $S_d$-invariant tensors in $V^{\otimes d}$.
  For $\lambda=(\lambda_1,\dotsc,\lambda_l)$, let $\Gamma^\lambda(V)=\Gamma^{\lambda_1}(V)\otimes\dotsb\otimes\Gamma^{\lambda_l}(V)$.
  The functors $\Gamma^\lambda$, as $\lambda$ runs over all partitions of $d$, generate $\Rep\mathbf\Gamma^d$ \cite[Proposition~2.9]{MR3077659}.
  Therefore, it suffices to prove the theorem for the functors $F=\Gamma^\lambda$.
  Since $\Gamma^\lambda$ is a subobject of $\otimes^d$ in the $\Rep\mathbf\Gamma^d$, $\Res^d\Gamma^d$ is a subobject of $\Res^d\otimes^d$ in $\fimod$.
  By the Noetherian property of $\FI$-modules \cite[Theorem A]{CEFN}, it suffices to show that $\Res^d\otimes^d$ is finitely generated.
  But $\Res^d\otimes^d(A)=\F^1(A)^{\otimes d} \cong \F^1(A^d)$, which is finitely generated in degree $d$ by \cite[Proposition~2.3.6]{MR3357185}.
\end{proof}
\subsection*{Acknowledgements}
We thank Aprameyo Pal, K. N. Raghavan, Anne Schilling, and S. Viswanath for helpful discussions and encouragement. We thank Brian Hopkins for his answer on mathoverflow \linebreak \url{https://mathoverflow.net/q/340231} which helped us prove part (1) of Theorem~\ref{theorem:2row2colhook}. We thank Dipendra Prasad for an argument in the proof of the sufficiency of the condition in part (3) of Theorem~\ref{theorem:2row2colhook}. We thank the referees for their comments and corrections to the original manuscript.

\end{document}